\documentclass[12pt]{amsart}
\usepackage[colorlinks=true,citecolor=blue,linkcolor=blue, backref=page]{hyperref}
\usepackage{amssymb,verbatim,amscd,amsmath,graphicx,nccmath}
\usepackage{enumerate}
\usepackage{graphicx}
\usepackage{caption}
\usepackage{subcaption}
\usepackage{pdfsync}
\usepackage{color,colortbl}
\usepackage{fullpage}
\usepackage{bbm}
\usepackage{comment}
\usepackage{multirow}
\usepackage{enumitem}
\usepackage{comment}
\usepackage{cleveref}
\numberwithin{equation}{section}
\newtheorem{introthm*}{Main Result}
\newtheorem{theorem}{Theorem}[section]
\newtheorem{lemma}[theorem]{Lemma}
\newtheorem{proposition}[theorem]{Proposition}
\newtheorem{corollary}[theorem]{Corollary}

\theoremstyle{definition}
\newtheorem{definition}[theorem]{Definition}

\newtheorem{def-prop}[theorem]{Definition-Proposition}

\newtheorem{remark}[theorem]{Remark}
\newtheorem{example}[theorem]{Example}

\newtheorem*{acknowledgment}{Acknowledgments}
\newtheorem{notation}[theorem]{Notation}
\newtheorem{question}[theorem]{Question}

\DeclareMathOperator{\Ass}{Ass}

\newcommand{\ceil}[1]{\lceil #1 \rceil}
\newcommand{\floor}[1]{\lfloor #1 \rfloor}

\def\pp{{\frak p}}

\def\1{{\bf 1}}
\def\0{{\bf 0}}

\begin{document}
	
	\title{Resurgence number of matoridal configuration}
	
	\author{Haoxi Hu}
	\address{Tulane University, Department of Mathematics and Statistics,
		6823 St. Charles Avenue
		New Orleans, LA 70118, USA}
	
	\email{hhu5@tulane.edu}

	
	\begin{abstract}
This article gives a new upper bound for the resurgence number of symbolic powers of matroidal configuration in the following situations: the height of the matroidal configuration is big, or the height is small, and the corresponding simplicial complex of the matroidal configuration is \emph{peaked}. The \emph{Peaked simplicial complex} is a generalization of bipartite graph. We also give some examples that the remaining case for which the bound cannot be improved. Furthermore, We provide a new proof of the formula for the \emph{resurgence number} of a \emph{star configuration} using the methods developed in this article.

	\end{abstract}
	\maketitle
	\section{Introduction}

Given an ideal $I$ of a Noetherian ring $R$, we can always find its \emph{$m$-th symbolic power} as follows:

$$I^{(m)}=\bigcap_{\pp\in \Ass(I)} I^mR_\pp \cap R .$$

\vspace{3mm}

Symbolic powers have been extensively studied in commutative algebra, and one important approach to understanding their behavior is through containment problem, see \cite{BocciHarbourne2010,GERAMITA2013279,symofstarconf,CSOPI,guardo2013asymptotic,1,2,3,4,5,6,7,8,9,MCSASPOTI}. The containment problem starts at this question: when is symbolic power of an ideal $I^{(m)}$ contained in ordinary power of an ideal $I^r$? The celebrated result of uniform bound was proved by \cite[Hochster-Huneke]{CSOPI} and \cite[Ein-Lazarsfeld-Smith]{UBSPSV} in different settings of characteristic of the field gave an important answer:  Let $R$ be a regular ring containing a field, and let $I$ be any proper ideal of $R$, then $I^{(hn)} \subset I^{n}$ for all $n \geq 1$, where $h$ is the big height of $I$, i.e. the largest height of any associated primes of $I$. Subsequently, the question of containment emerged as a central problem in this field, then a new concept called \emph{resurgence number of symbolic powers}, denoted as $\rho(I)$, was introduced in \cite[Bocci-Harbourne]{RIPCP} as follows:

$$\rho(I) =  \sup \left\{ \frac{m}{r} ~|~ I^{(m)} \not\subseteq I^r, m\geqslant 1, r\geqslant 1 \right\}$$

\vspace{4mm}

Computing resurgence number is very difficult in general. Resurgence number is only known for \emph{star configuration of hypersurfaces} (or just \emph{star configuration} for abbreviation), see \Cref{remark: special cases of star config} for the history of computing resurgence number of star configuration.  In this paper, by different method, we will give the same formula to compute the resurgence number and the strict containment of \emph{uniform matroidal configuration} which includes case of \emph{star configuration}, see \Cref{theorem: resurgence number of generalized star configuration}.

The main subject we study in this article are \emph{$C$-matroidal ideal} and its corresponding \emph{$C$-matroidal configuration} (also called as \emph{matroidal configuration}). \emph{$C$-matoridal ideal} is a monomial ideals defined on a polynomial ring over a field, where supports of generators satisfy axioms of circuits of a matroid (see the Example \ref{quick example of matroidal ideal}), the corresponding \emph{matroidal configuration} has similar settings by replacing variables with a regular sequence consisting of homogeneous components, so \emph{$C$-matroidal configuration} builds the connection between \emph{uniform matroid} $U_{c,n}$ and \emph{star configuration} (see Example \ref{example: uniform matroid and star configuration}). We also want to mention that all results of \emph{$C$-matroidal ideal} in this article can be applied to its corresponding \emph{matroidal configuration} as well (see Remark \ref{C-matroidal ideal and matroidal configuration}).  \emph{$C$-matroidal ideal} is the usual definition of the matroidal ideal, but there is also another way to define matroidal ideal via bases of a matroid, we will call it \emph{$B$-matroidal ideal}, and \emph{$B$-matroidal ideal} is the \emph{Alexandar dual} of \emph{$C$-matroidal ideal} with respect to their corresponding simplicial complexes. We study the following question in this article.

\begin{question}
What is the \emph{resurgence number} of \emph{$C$-matroidal ideal}?
\end{question}

Our method to tackle this question is to study its corresponding simplicial complex structures of \emph{$C$-matroidal ideal}, see Definition \ref{def of matroidal ideal}, and we define a new class of simplicial complex called \emph{peaked simplicial complex}, which generalizes the bipartite graph (see Example \ref{example: peaked simplicial complex to bipartite graph}). When the height of the \emph{$C$-matroidal ideal} is big, or the height is small and the corresponding simplicial complex is \emph{peaked}, a new upper bound for the resurgence number is given, which is our second main result. Before our result, the upper bound of resurgence of \emph{matroidal ideal and configuration} is $c$ (the height of \emph{$C$-matroidal ideal}), and our new bound is $c-1$.

\begin{theorem}[\Cref{theorem: resurgence bound of matroidal configuration}]

let $k = \floor{\frac{n}{n-c+1}}$ where $n$ is the number of variables used in the $C$-matroidal ideal $I$, $c$ is the height of $I$. Then if $k > 1$, or $k = 1$ and the corresponding simplicial complex $\Delta$ is peaked, we have resurgence upper bound for $I$ as follows:
\begin{align*}
\rho(I) \leq c-1
\end{align*}

In particular, this result can be also applied to corresponding matroidal configuration.

\end{theorem}

The bound will fail easily if $k=1$ and $\Delta$ is not \emph{peaked}, see the simple Example \ref{example: non-peaked condition fails}. This explains why the case of a \emph{non-peaked} simplicial complex with $k=1$ is the only remaining case for which the bound cannot be improved.

\vspace{1mm}

\begin{acknowledgment}
We honorably thank professor Tewodros Amdeberhan for giving an idea to us to prove Lemma \ref{lemma: monomial assignment}. We are also very thankful to professor Paolo Mantero and Doctor Vinh Nguyen for some helpful discussions.
\end{acknowledgment}

\vspace{3mm}

\section{Preliminary}\label{section: Preliminary}

\subsection{Matroid}

Bases and circuits are two important concepts which characterize and also define matroids, so they lead to two types of matroidal ideal, we will only focus on one type. Before we get into details, let's we recall some definitions:

\begin{definition}

The basis $\mathcal{B}$ of a matroid $M$ defined on a ground set $E$ is a nonempty collection of subsets of $E$ satisfies the following properties:
\begin{itemize}

\item[(1)] (Basis exchange property) For any $B_1, B_2 \in \mathcal{B}$, and for any $v \in B_1 \backslash B_2$, there exist a $w \in B_2 \backslash B_1$ such that $(B_1 - v) \cup \{ w \} \in \mathcal{B}$.

\item[(2)] (Symmetric basis exchange property) For any $B_1,B_2 \in \mathcal{B}$ and for any $v \in B_1 \backslash B_2$, there exist a $w \in B_2 \backslash B_1$ such that both $(B_1 - v) \cup \{ w \} \in \mathcal{B}$ and $(B_2 - w) \cup \{ v \} \in \mathcal{B}$.

\item[(3)] (Symmetric multi-basis exchange property) For any $B_1, B_2 \in \mathcal{B}$, and for any subset $A \subseteq B_1 \backslash B_2$, there exist a subset $B \subseteq B_2 \backslash B_1$ such that both $(B_1 - A) \cup B \in \mathcal{B}$ and $(B_2 - B) \cup A \in \mathcal{B}$.

\end{itemize}

\end{definition}

Note that \emph{basis} of matroid is a collection of maximal independent sets of the matroid.

\begin{definition}

A collection of \emph{circuits}, denoted as $\mathcal{C}$, of a matroid $M$ defined on a ground set $E$ is a nonempty collection of subsets of $E$ satisfies the following properties:

\begin{itemize}

\item[(1)] (Nonempty) $\emptyset \notin \mathcal{C}$.

\item[(2)] (Minimality) For any $C_1,C_2 \in \mathcal{C}$, if $C_1 \subseteq C_2$, then $C_1 = C_2$.

\item[(3)] (Circuit elimination) For any distinct $C_1, C_2 \in \mathcal{C}$, if $e \in C_1 \cap C_2$, then there exists $C_3 \in \mathcal{C}$ such that $C_3 \subseteq (C_1 \cup C_2) - e$.

\end{itemize}

\end{definition}

\begin{definition}

Let matorid $M = (E,\mathcal{B})$ defined by the ground set $E$ and basis $\mathcal{B}$, then the \emph{dual} of $M$, denoted as $M^*$, is defined by its basis $\mathcal{B}^*$ as the following:
\begin{align*}
B^* = \{ E \backslash B: \, B \in \mathcal{B} \}
\end{align*}

\end{definition}

Note that $M^*$ is also a matriod, and $M^{**} = M$. A \emph{circuit} of matroid is a minimal dependent set of the matroid in terms of inclusion. A \emph{cocircuit} is a \emph{circuit} of the \emph{dual} of the matroid.

\subsection{C-Matroidal ideal and Configuration}\label{subsection: C-matroidal ideal and configuration}

\begin{definition}\label{def of matroidal ideal}\cite[definition 2.10]{matoridalsympow}
A squarefree monomial ideal $I \subset T$ is \emph{$C$-matroidal} if $I$ is the \emph{cover ideal} of a matroid associated to a simplicial complex $\Delta$. More specifically, given a matroid $M$ associated with simplicial complex $\Delta$, the support of a generator of \emph{Stanley–Reisner ideal} is exactly the \emph{circuit} of $M$, so list of supports and list of circuits are one-to-one correspondence. If we take the \emph{cover ideal} of $\Delta$, the support of a generator of \emph{cover ideal} will be \emph{cocircuit} of $M$. See the Example \ref{quick example of matroidal ideal} for the construction.

Recall the \emph{cover ideal} of a simplicial complex $\Delta$ is $J(\Delta) = \bigcap_{F \in \mathcal{F}(\Delta)} \mathfrak{p}_F$ where $\mathcal{F}(\Delta)$ is the collection of all facets of $\Delta$ and $\mathfrak{p}_F = (x_i \, : \, i \in F)$. the \emph{Stanley–Reisner ideal} of $\Delta$ is $I_{\Delta} = \bigcap_{F \in \mathcal{F}(\Delta)} \mathfrak{p}_{[n]-F}$ where $[n]$ is the ground set of $\Delta$.

\end{definition}

\begin{remark}\label{more defs of matroidal ideal}

There are more equivalent conditions to define \emph{$C$-matroidal ideal} and its construction related to another way to define matroidal ideal (see \cite[definition 2.10]{matoridalsympow}), but here we focus on Definition \ref{def of matroidal ideal}.

\emph{$C$-matroidal ideal} is the usual definition of \emph{matroidal ideal}, while another common definition of \emph{matroidal ideal} is given via \emph{bases} of the matroid, which is exactly the \emph{facet ideal of the matroid}. Therefore, the support of $B$-matroidal ideal the generators are \emph{bases}. Note that $B$-matroidal ideal is the \emph{Alexandar dual} of \emph{$C$-matroidal ideal}.

\end{remark}

\begin{example}\label{quick example of matroidal ideal}

Given a simplicial complex $\Delta$ defined on a ground set $\{ 1, 2, 3, 4 \}$, and its facets are $\{ \{1,2\}, \{1,3\}, \{2,3\}, \{2,4\}, \{3,4\}  \}$, take the cover ideal of these facets, we have a collection of generators $\{  124, 134, 23 \}$, which satisfy the axioms of circuits of matroid, so this a $C$-matroidal ideal. Take the \emph{Alexandar dual} of this ideal, we have a new ideal with generators $\{ 12,13,23,24,34 \}$, and they satisfy the axiom of bases of matroids, so the new ideal is \emph{$B$-matroidal ideal}.

\end{example}

\begin{definition}\cite[definition 2.1]{matoridalsympow}
Let $R$ be a polynomial ring over a field, and let $\mathcal{F} = \{ f_1, \, ... \, , f_s  \}$ be our collection of homogeneous components of $R$. Let T be a polynomial ring over the same field of $s$ variables such that the ring homomorphism $\phi: T \longrightarrow R$ sending variable $x_i$ of $T$ to $f_i$ for all homogeneous components. Let $I \subset T$ be any monomial ideal, we denote $I_*$ for the ideal generated by $\phi(G(I))$ where $G(I)$ is the collection of minimal generators of $I$. Now we define \emph{matroidal configuration}:

Fix $I$ be a $C$-matoidal ideal in $T$, let $c = ht(I)$, and assume any $c+1$ of $f_i$'s form a regular sequence. Then $I_*$ is called a specialization of $I$, or \emph{the defining ideal of a $C$-matroidal configuration (of hypersurfaces)}, for abbreviation, just \emph{matroidal configuration}. If all $f_i$'s are linear and $dim(R) = c+1$, we call $I_*$ is \emph{the defining ideal of a matroidal configuration of points}.

We want to emphasize that since $I$ is any monomial ideal, there might be some variables not used as monomials for $I$, so we let $n$ be the number of elements of union of all facets of a simplicial complex related to $I$, i.e. support of all generators of $I$. Note that $c \leq n \leq s$. Similarly, for the corresponding matroidal configuration $I_*$, $n$ is the support of all generators with respect to homogeneous components.
\end{definition}

\begin{remark}\label{C-matroidal ideal and matroidal configuration}

We should mention that $(I_*)^{r} = (I^{r})_*$ for all $r \leq 1$ for any monomial ideal $I$. The celebrating result from \cite[theorem 3.6(1)]{MCSASPOTI} shows that $(I_*)^{(m)} = (I^{(m)})_*$ for any \emph{$C$-matroidal ideal} $I$ and corresponding \emph{matroidal configuration} $I_*$ generated by monomial forms of $\mathcal{F}$ defined above.

\end{remark}

\begin{example}\label{example: uniform matroid and star configuration}

The \emph{generalized uniform matroid} $U_{c,n}$ is defined by ground set with $s$ elements. The union of its independent sets contain $n$ elements, while the independent sets are exactly all possible subsets of size at most $c$. Our definition of uniform matorid slightly generalizes the usual definition of uniform matroid in which $n$ is always $s$. It is not hard to see that the \emph{defining ideal of a matroidal configuration} of the usual definition of uniform matroid is exactly the \emph{star configuration of hypersurfaces}. See \cite[example 3.4]{MCSASPOTI} for the original idea, it also explains that it satisfies the definition via circuits of a matroid.

\end{example}

\subsection{Peaked Simplicial Complex}

It is known that $\Delta$ is a \emph{squarefree glicci simplicial complex} (see \cite{MCSASPOTI}, and see \cite{NAGEL20082250} for the definition of \emph{glicci simplicial complex}). However, the way to construct such simplicial complex does not help us a lot for ideal containment problems, so we want to classify $\Delta$ based on a different setting, where our definition of \emph{peaked simplicial complex} comes into play.

\begin{definition}

Given a simplicial complex $\Delta$ defined on vertex set $[n]$, let $\mathcal{F}_i(\Delta)$ be the collection of all facets of $\Delta$ that contain $\{i\}$. We call $\Delta$ is \emph{peaked} if there exist a strict subset $A \not\subseteq [n]$ such that the collection of all facets satisfies $ \cup_{i \in A} \mathcal{F}_i(\Delta) = \mathcal{F}(\Delta)$, and these $\mathcal{F}_i$'s are disjoint.

\end{definition}

\begin{example}\label{example: peaked simplicial complex to bipartite graph}

When all the facets of a \emph{peaked simplicial complex} are either edges or vertices, then it's exactly a bipartite graph. Considering $A$ as one of the two partitions, there is no edge connecting any two vertices in $A$, say $i$ and $j$, otherwise $\mathcal{F}_i(\Delta)$ and $\mathcal{F}_j(\Delta)$ will share edges. Because $ \cup_{i \in A} \mathcal{F}_i(\Delta) = \mathcal{F}(\Delta)$, there is no edge connecting any two vertices in the other partition as well.

\end{example}

\begin{example}

We can look at the \emph{peaked simplicial complex} in more general case. If there are facets with dimension greater than $0$, take one such element, say $i$. We can consider each facet of $\mathcal{F}_i$ is a ``wall" of a ``pyramid-like" shape with ``peak vertex" at $\{i\}$. Since all $\mathcal{F}_j$ are disjoint where $j \in A$, the \emph{peaked simplicial complex} consists of multiple pyramid-like shapes, that is how we gave the name of this type of simplicial complexes.

\end{example}

\begin{notation}

The notations $\Delta$, $T$, $I$, $I_*$, $\mathcal{F}$, $\mathcal{F}_i$, $c$, $s$, $n$, and $U_{r,n}$ we given from the definitions and examples above will be used throughout the rest of the paper unless otherwise is stated.

\end{notation}

\vspace{3mm} 

\section{Proof of the main results}\label{section: proof}

Our method to show the main results is to study the properties of circuits of matroid and facets of corresponding simplicial complex, then we describe the combinatorial structures of the squarefree parts of the supports of the \emph{$C$-matroidal ideal}. \cite[Theorem 3.7]{matoridalsympow} (our Theorem \ref{theorem: formula for minimal generators of sym pow of star config}) describes the generators of symbolic powers of $C$-matroidal ideals, and \cite[Proposition 3.19]{matoridalsympow} (our Proposition \ref{proposition: description of squarefree parts}) describes the squarefree parts of first $c$ symbolic powers via simplicial complex. Following Theorem \ref{theorem: formula for minimal generators of sym pow of star config}, we develop a formula saying that all symbolic powers are generated by the first $c$ symbolic powers, which is Corollary \ref{corollary: binomial expansion of sym pow of star config}. Moreover, following the Proposition \ref{proposition: description of squarefree parts}, we give more details of the structure of corresponding simiplicial complex, which are more applicable for solving problems, and it is our Proposition \ref{prop: combinatorial correspondence of squarefree parts}.

Denote squarefree part of $I^{(m)}$ as $SF(I^{(m)})$. As we mentioned in Remark \ref{C-matroidal ideal and matroidal configuration}, results of $C$-matroidal ideal $I$ can be applied to matroidal configuration $I_*$, so similarly, when we talk about ``degree" matroidal configuration, the degree for $I_*$ is graded with respect to homogeneous components from $\mathcal{F}$.
	
	\begin{theorem}\label{theorem: formula for minimal generators of sym pow of star config}
		\cite[Theorem 3.7]{matoridalsympow} Let $I$ be the usual settings. For any $m \geq 1$, a minimal generating set of $I$ is given by 

\begin{center}

$G_{I^{(m)}} = \Bigl\{$\begin{tabular}{m{33mm} | m{90mm}}
$M = M_{(1)} \cdots M_{(t)}$ & $M_{(i)} \in G_{I^{(c_i)}}$ and is squarefree, where $1 \leq c_i \leq ht(I)$ with $\sum c_i = m$ and $supp(M_{(1)}) \supset ... \supset supp(M_{(t)})$
\end{tabular} $\Bigr\}$

\end{center}
		
	\end{theorem}
	
\medskip
	
The following is the Corollary of \cite[Theorem 3.7]{matoridalsympow}, and it is also the "refined" version of \cite[Theorem 3.9]{matoridalsympow}.
	
\begin{corollary}\label{corollary: binomial expansion of sym pow of star config} 

Using usual settings $I$. Then we have the binomial expansion as follows:

\begin{center}
$I^{(m)} = \sum\limits_{i=1}^{c} I^{(i)} \, I^{(m-i)}$ for $m > c$
\end{center}

In other words, symbolic powers of matroidal ideal are generated by the first $c$ symbolic powers.

\end{corollary}

\begin{proof}
First know that $I^{(m)} \supset \sum_{i=1}^{c} I^{(i)}I^{(m-i)}$, we just need to show the inclusion. 

Let $M = M_{(1)}...M_{(t)}$ be the generator in $G_{I^{(m)}}$. By definition, $M_{(1)} \cdots M_{(t-1)}$ and $M_{(t)}$ are generators from the lower symbolic powers and their powers sum up to $m$. If we write $M_{(t)}$ as product of squarefree monomial factors, which violates the support inclusion rule in theorem \ref{theorem: formula for minimal generators of sym pow of star config} since $M_{(t)}$ itself is already squarefree, and we also know $M_{(t)}$ comes from $I^{(t)}$ where $1 \leq t \leq c$.
\end{proof}

Theorem \ref{corollary: binomial expansion of sym pow of star config} and \cite[Theorem 3.9]{matoridalsympow} imply all squarefree generators come from first $c$ symbolic powers. First we need to describe squarefree parts to help us to discuss the ideal containment problem, so we need to introduce the following proposition:

\begin{proposition}\cite[Proposition 3.19]{matoridalsympow}\label{proposition: description of squarefree parts}
Let $J$ be a squarefree monomial ideal, and let $\Delta$ is a simplicial complex such that $J = J(\Delta)$ (the cover ideal of $\Delta$), and $\mathcal{F}(\Delta_m) = \{ F - A \, : \, F \in \mathcal{F}(\Delta), \, A \subset F, \, |A| = m - 1 \}$. Then 
\begin{align*}
SF(J^{(m)}) = J(\Delta_m)
\end{align*}

\end{proposition}

\medskip

To reveal more combinatorial structures of this proposition, we need the following lemma:

\begin{lemma}\label{lemma: compute intersection of ideals}

Let $R$ be the polynomial ring of $n$ variables where $c \leq n$. Denote $L_c$ as the collection of ideals generated by $c$ distinct variables of $R$, i.e. the element of $L_c$ is the ideal of the form $( x_{i_1}, x_{i_2}, \, ... \, , x_{i_c} )$ where $x_{i_j}$'s are all distinct. Then the intersection of all elements of $L_c$ is an ideal (minimally) generated by all possible squarefree elements of degree $n+1-c$.

\end{lemma}

\begin{proof}
If two ideals share the same generators, then the intersection of these two ideals preserve corresponding generators. Without loss of generality, consider $x_1$ be in one element $l_1$ of $L_c$, finding a new element $l_2$ of $L_c$ such that it contains $x_2$ as a generator, then the intersection of these two ideals contains generator $x_1 x_2$. We can keep doing this until we have $l_{n-c}$, and consider $ l = ( x_{n-c+1}, \, ... \, , x_{n} ) \in L_c$, then $L = (\cap_{i=1}^{n-c} l_i) \cap l$ will contain list of generators $ G = \{ x_1 x_2 \cdots x_{n-c} x_j \, : \, \text{for all} \; n-c+1 \leq j \leq n \}$ as generators.

Take any element $l'$ from $L_c$ outside $\{ l_1, ... \, , l_{n-c}, l \}$, consider $l'$ intersects with $L$, we need to compare generators from $l'$ with generators from $L$. Let $x_k$ be a generator of $l'$, if $k \leq n-c$, then the intersection ideal will still contain $G$ as generators. If $k > n-c$, it follows that $x_1 x_2 \cdots x_{n-c} x_k$ is a generator of the intersection ideal, as well as $x_1 x_2 \cdots x_{n-c} x_j x_k$ where $j \neq k$. However, the latter generators are all redundant since we already had $x_1 x_2 \cdots x_{n-c} x_k$ as our generator. Note that there are some generators of $l'$ are $x_k$'s such that $k \leq n-c$, otherwise $l' = l$. Therefore, the intersection ideal will still have list of generators $G$.

Since the previous process of constructing $L$ is arbitrary with respect to choices of variables as generators, the generators of intersection of all elements of $L_c$ are all possible elements of degree $n+1-c$ because generators of $l$ can be any choice of $c$ subset of $[n]$.
\end{proof}

\begin{proposition}\label{prop: combinatorial correspondence of squarefree parts}

Let $I$ be an intersection ideal of ideals corresponding to $(c-1)$-faces from a simplex $F$ of dimension $n-1$. Then squarefree parts of $m$ symbolic power will corresponds to $(n+m-c-1)$-faces of $F$, i.e. $SF(I^{(m)}) = \{ x_{i_1} \cdots x_{i_{n+m-c}} : \, \{i_1, ... \, , i_{n+m-c}\} \subset F \}$. And $I$ is also a $C$-matroidal ideal.

\end{proposition}

\begin{proof}
First note that $I$ is a squarefree monomial ideal followed from the simplicial complex structure. By Proposition \ref{proposition: description of squarefree parts} and Lemma \ref{lemma: compute intersection of ideals}, it will be generated by all squarefree monomials in $\{ x_i : \, i \in F \}$ of degree $n+1-c$, this is an one-to-one correspondence to $(n-c)$-faces of $F$, then by the definition, $I$ is also a $C$-matroidal ideal. Similarly, squarefree parts of $m$ symbolic power will corresponds to $(n+m-c-1)$-faces of $F$. 
\end{proof}

\begin{proposition}\label{prop: peak and containment}

Using the usual settings $\Delta$, $I$, $c$, $n$, $I_*$, and let $k = \floor{\frac{n}{n-c+1}}$. We have the following properties:
\begin{itemize}
\item[(1)] If $k > 1$, then $I^{2} \supset I^{(c)}$, and $I^{\ceil{\frac{m}{c-1}}} \supset I^{(m)}$ for all $m > c$.

\item[(2)] If $k = 1$, and $\Delta$ is peaked, then $I^{2} \supset I^{(c)}$, and $I^{\ceil{\frac{m}{c-1}}} \supset I^{(m)}$ for all $m > c$.
\end{itemize}

In particular, the ideal containment results can be also applied to corresponding matroidal configuration $I_*$.

\end{proposition}

\begin{proof}
$(1)$ Let $F$ be a simplex of dimension $n-1$ defined on the same vertex set of $\Delta$. Following the proof Lemma \ref{lemma: compute intersection of ideals} and Proposition \ref{prop: combinatorial correspondence of squarefree parts}, let $\tilde{I}$ be the intersection ideal of all $(c-1)$-faces from $F$, so $\tilde{I}$ is $I$ by removing some face ideals from the ideal intersection. Observe that every generator of $I$ divide some generators of $\tilde{I}$, and also every generator of $\tilde{I}$ come from some generators from generators of $I$ times some squarefree monomials.

By Proposition \ref{prop: combinatorial correspondence of squarefree parts}, the degree of generators of $\tilde{I}$ is $n-c+1$, then by assumption, we can always find $k$ generators of $\tilde{I}$ such that they don't share any common divisors and the product of them divides $x_1 \cdots x_n$. Note that $x_1 \cdots x_n$ is the only squarefree part of $I^{(c)}$. This implies we can also find $k$ generators from $I$ that do the same things.

By Theorem \ref{theorem: formula for minimal generators of sym pow of star config}, we have $I^2$ contains $I^{(c)}$ since the non-squarefree generator of $I^{(c)}$ come from multiple generators from lower symbolic powers. Furthermore, following the same reason, and $\ceil{\frac{m}{c-1}} \leq k \floor{\frac{m}{c}}$ for all $m > c$, we will have $I^{\ceil{\frac{m}{c-1}}}$ contains $I^{(m)}$ for all $m > c$.

\vspace{2mm}

$(2)$ When $\Delta$ is peaked, without loss of generality, let $\{ \mathcal{F}_1, ... \, , \mathcal{F}_k  \}$ be the collection satisfying the peaked condition. By ideal intersection, observe that $x_1 \cdots x_k$ and $x_{k+1} \cdots x_n$ are generators of $I$. The product of these two generators divides $x_1 \cdots x_n$, again we can follow the proof of Proposition \ref{prop: peak and containment}(1), the proof is complete.
\end{proof}

\begin{theorem}\label{theorem: resurgence bound of matroidal configuration}

let $k = \floor{\frac{n}{n-c+1}}$ where $n$ is the number of components used in the $C$-matroidal ideal $I$, $c$ is the height of $I$. Then if $k > 1$, or $k = 1$ and the corresponding simplicial complex $\Delta$ is peaked, we have resurgence upper bound for $I$ as follows:
\begin{align*}
\rho(I) \leq c-1
\end{align*}

In particular, this result can be also applied to corresponding matroidal configuration $I_*$.

\end{theorem}

\begin{proof}
Following proposition \ref{prop: peak and containment}, for all $m \geq c$, $I^r \not\supset I^{(m)}$, then it must satisfy $r > \ceil{\frac{m}{c-1}}$. Note that when $m = 2$, $\ceil{\frac{m}{c-1}} = 2$, it will exclude the situation that $I^2 \supset I^{(c)}$. It is equivalent to have the following:
\begin{align*}
\frac{m}{r} < \frac{m}{\ceil{\frac{m}{c-1}}} \leq c-1
\end{align*}

For all $m < c$, $I \supset I^{(m)}$, in this case, $I^r \not\supset I^{(m)}$, it must satisfy the following:
\begin{align*}
\frac{m}{r} < m \leq c-1
\end{align*}

Combining two cases, we have $\rho(I) \leq c-1$, the result can be applied to $I_*$ since the ideal containment is preserved.
\end{proof}

If $k=1$ and $\Delta$ is not \emph{peaked}, the bound in Theorem \ref{theorem: resurgence bound of matroidal configuration} will fail easily. Here is the simple example:

\begin{example}\label{example: non-peaked condition fails}

Considering a simplicial complex defined on vertex set $\{ 1, 2, 3, 4 \}$ with facets $\{ \{1,2\},\{1,3\},$ $\{1,4\},\{2,3\}  \}$, obviously it is not \emph{peaked}. The generators of the corresponding $C$-matroidal ideal $I$ is $\{ x_1 x_3, x_1 x_2 x_4, x_2 x_3 x_4 \}$, and the height $c=2$. We need the resurgence number of $I$ is greater than $c-1$, equivalently it means that there exists $m$ and $r$ such that $\frac{m}{r} > c-1$ and $I^{(m)} \not\subseteq I^r$. For this example, we need $m > r$ and $I^{(m)} \not\subseteq I^r$. It's known that the only squarefree part of $c$ symbolic power of the $C$-matroidal ideal is always $x_{i_1} \cdots x_{i_n}$, so the only squarefree part of $I^{(2)}$ is $x_1 x_2 x_3 x_4$. We want to show that $I^{(8)} \not\subseteq I^7$. Note that $x_1^4 x_2^4 x_3^4 x_4^4 \in I^{(8)}$, but it is also the smallest degree term in $I^6$ that divides $x_1^4 x_2^4 x_3^4 x_4^4$, so it is impossible that $I^7$ contains $x_1^4 x_2^4 x_3^4 x_4^4$. 

\end{example}

\begin{theorem}\label{theorem: resurgence number of generalized star configuration}

Let $I$ be the $C$-matroidal ideal corresponding to uniform matroid $U_{c,n}$. Then we have the followings:

\begin{itemize}

\item[(1)] The strict containment:
\begin{align*}
I^{r} \; \text{contains} \; I^{(m)} \; \text{if only if} \; r \leq \frac{\ceil{\frac{m}{c}} (n-c) + m}{n-c+1}
\end{align*}

\item[(2)] And the resurgence number:
\begin{align*}
\rho(I) = \frac{c(n-c+1)}{n}
\end{align*}

In particular, if we apply the formula to its corresponding matroidal configuration, let $n=s$ where $s$ is the number of homogeneous components, then it is also the resurgence number of the star configuration.
\end{itemize}

\end{theorem}

\bigskip

It's not hard to see that the simplicial complex corresponding to \emph{generalized uniform matroid} is not \emph{peaked}, but it does satisfy a nice property, which is the following lemma.

\begin{lemma}\label{lemma: monomial assignment}

Let $M$ be a momomial defined on a finite set of variables $\{x_1, ... , x_n \}$, and $n$ can be arbitrary big. Suppose degree of each variable of $M$ is at most $m$. If there exists $a$,$b$ such that $deg(M) = ab$ and $b \geq m$, then $M$ can be factored into $b$ squarefree monomials of same degree $a$.

\end{lemma}

\begin{proof}

Let $M = x_1^{p_1} x_2^{p_2} ... x_n^{p_n}$. Without loss of generality, we might assume $p_1 \geq p_2 \geq ... \geq p_n$. We will prove the statement by recursive process.

When $b=1$, $M$ is squarefree, then we are done.

If $b \geq 2$, we can observe the following facts:

\begin{itemize}

\item[(1)] $a \leq n$. Because $deg(M) = ab \leq nb$.

\item[(2)] $p_a \geq 1$. If $p_a = 0$, $deg(M) = p_1 + p_2 + ... p_a \leq (a-1)b < ab$.

\item[(3)] $p_{a+1} \leq b-1$. If $p_{a+1} \geq b$, $deg(M) \geq p_1 + p_2 + ... + p_{a+1} = (a+1)b > ab$.

\end{itemize}

Let $N = x_1 x_2 ... x_a$ be our new factor, we can always do this by $(1)$ and $(2)$, and note that the variables of remaining monomial $M/N$ are at most $b-1$ by $(3)$. Now $deg(M/N) = a(b-1)$, apply the above factorization recursively by $b-1$, $b-2$, ... , $1$, we will get the factorization we want.

\end{proof}

\begin{proof}[Proof of Theorem \ref{theorem: resurgence number of generalized star configuration}]

First by Proposition \ref{prop: combinatorial correspondence of squarefree parts} and the definition of uniform matroid, $I$ is exactly the intersection ideal of all $(c-1)$-faces from a simplex $F$ of dimension $n-1$, then the $SF(I^{(m)})$ corresponds to $(n+m-c-1)$-faces of $F$. Then by Theorem \ref{theorem: formula for minimal generators of sym pow of star config} and Corollary \ref{corollary: binomial expansion of sym pow of star config}, every symbolic power can be represented by the sum of products of squarefree parts of first $c$ symbolic power by inductive process. Observe that the lowest degree of  generators of $I^{(m)}$ is $\ceil{\frac{m}{c}}(n-c) + m$ because the lowest degree of the generator can be written as $\ceil{\frac{m}{c}}$ product of squarefree generators from first $c$ symbolic powers. For simplicity, let it be $M^{\ceil{\frac{m}{c}}} = M_{(1)} M_{(2)} ... M_{(\ceil{\frac{m}{c}})}$.

Note that $\ceil{\frac{m}{c}}(n-c+1) = \ceil{\frac{m}{c}} (n-c) + \ceil{\frac{m}{c}} \leq \ceil{\frac{m}{c}} (n-c) + m$ , we want to show that for all $r \geq \ceil{\frac{m}{c}}$ such that $r(n-c+1) \leq \ceil{\frac{m}{c}}(n-c) + m$, we have $I^r$ contains $I^{(m)}$.

If $r(n-c+1) = \ceil{\frac{m}{c}}(n-c) + m$, then applying Lemma \ref{lemma: monomial assignment}, the generators in $I^r$ generate $M^{\ceil{\frac{m}{c}}}$.

If $r(n-c+1) < \ceil{\frac{m}{c}}(n-c) + m$. Since $deg(M_{(i)}) \geq n-c+1$ and $\sum_{i=1}^{\ceil{\frac{m}{c}}} deg(M^{(i)}) = \ceil{\frac{m}{c}}(n-c) + m$. Now considering to remove variables from some $M_{(i)}$'s such that their degrees are greater and equal to $n-c+1$, and also make the corresponding new monomial $\overline{M}^{\ceil{\frac{m}{c}}}$ having degree exactly $r(n-c+1)$. Apply Lemma \ref{lemma: monomial assignment} again, generators of $I^r$ generate $\overline{M}^{\ceil{\frac{m}{c}}}$, then so $M^{\ceil{\frac{m}{c}}}$. 

Note that the higher degree generators of $I^{(m)}$ are products of $\ceil{\frac{m}{c}} + i$ squarefree parts of first $c$ symbolic powers of $I$, denoted as $M_{\ceil{\frac{m}{c}}+i}$, where $1 \leq i \leq m - \ceil{\frac{m}{c}}$ by Corollary \ref{corollary: binomial expansion of sym pow of star config}. Furthermore, $deg(M^{\ceil{\frac{m}{c}}+i}) = (\ceil{\frac{m}{c}}+i)(n-c) + m > \ceil{\frac{m}{c}}(n-c) + m$. By the same methods above, there must exists $r' \geq r$ such that $I^{r'}$ contains $M^{\ceil{\frac{m}{c}}+i}$, so $I^r$ contains the higher degree generators. Therefore we only need to focus on the lowest degree generators, it follows that $I^r$ contains $I^{(m)}$. However, if $r(n-c+1) > \ceil{\frac{m}{c}}(n-c) + m$, $I^r$ won't contain the lowest degree generators of $I^{(m)}$. In conclusion, we have the following strict containment:
\begin{align*}
I^{r} \text{contains} \; I^{(m)} \; \text{iff} \; r \leq \frac{\ceil{\frac{m}{c}} (n-c) + m}{n-c+1}
\end{align*}

Furthermore, for all $I^{r} \not\supset I^{(m)}$, it must satisfy the following:
\begin{align*} 
\frac{m}{r} < \frac{n-c+1}{\frac{1}{m} \ceil{\frac{m}{c}} (n-c) + 1}
\end{align*}

Observe that supremum of right side of the inequality is the resurgence number, which is the following:
\begin{align*}
\rho(I) = \sup_{m \geq 1} \biggl\{ \frac{n-c+1}{\frac{1}{m} \ceil{\frac{m}{c}} (n-c) + 1} \biggl\} = \frac{n-c+1}{\frac{1}{m} \cdot \frac{m}{c} (n-c) + 1} = \frac{c(n-c+1)}{n}
\end{align*} 

\end{proof}

\begin{remark}\label{remark: special cases of star config}
\cite[Theorem 4.8]{MCSASPOTI} gave the formula to compute resurgence number of star configuration under the condition that all homogenous components have same degree, while \cite[Theorem 2.4.3(a)]{BocciHarbourne2010} was about star configuration of points on $\mathbb{P}^n$. Eventually,\cite{LAMPABACZYNSKA20155402} derived the formula for \emph{star configuration of hypersurfaces}, and it is equivalent to the formula in \Cref{theorem: resurgence number of generalized star configuration}.

\end{remark}

	\bibliographystyle{alpha}
	\bibliography{References}

\end{document}